\documentclass{article}
\usepackage{geometry,enumerate}
\usepackage{amsmath, amssymb,amsthm,mathtools}
\usepackage[all]{xy}
\usepackage{graphics}
\usepackage{comment}
\usepackage{cleveref}
\usepackage{color}
\usepackage{multicol}\columnseprule=0.4pt
\usepackage{sectsty}
\title{Generalizations of the relationship between \\ quasi-hereditary algebras and directed bocses}
\author{Goto Yuichiro}
\date{}
\newtheorem{df}{Definition}[section]
\newtheorem{eg}[df]{Example}
\newtheorem{thm}[df]{Theorem}
\newtheorem{prop}[df]{Proposition}
\newtheorem{lem}[df]{Lemma}

\newtheorem{rmk}[df]{Remark}

\newcommand{\A}{\mathcal{A}}
\newcommand{\B}{\mathcal{B}}
\newcommand{\Cat}{\mathcal{C}}
\newcommand{\F}{\mathcal{F}}
\newcommand{\LL}{\mathcal{L}}
\newcommand{\QQ}{\mathcal{Q}}
\newcommand{\D}{\mathbb{D}}
\newcommand{\Z}{\mathbb{Z}}
\newcommand{\mbP}{\mathbb{P}}

\newcommand{\HH}{H}
\newcommand{\oDelta}{\overline{\Delta}}
\DeclareMathOperator{\tensor}{\otimes}
\DeclareMathOperator{\Hom}{Hom}
\DeclareMathOperator{\id}{id}
\DeclareMathOperator{\rad}{rad}

\DeclareMathOperator{\kernel}{Ker}
\DeclareMathOperator{\image}{Im}
\DeclareMathOperator{\module}{mod}
\DeclareMathOperator{\add}{add}
\DeclareMathOperator{\End}{End}
\DeclareMathOperator{\Aut}{Aut}
\DeclareMathOperator{\Ext}{Ext}
\DeclareMathOperator{\op}{op}
\DeclareMathOperator{\twmod}{twmod}
\DeclareMathOperator{\conv}{conv}

\Crefname{subsection}{Subsection}{Subsections}
\Crefname{thm}{Theorem}{Theorems}
\Crefname{lem}{Lemma}{Lemmas}

\begin{document}

\maketitle \vspace{3mm}

\begin{abstract}
  Koenig, K\"ulshammer and Ovsienko showed that Morita equivalence classes of quasi-hereditary algebras are in one-to-one correspondence with equivalence classes of the module categories over directed bocses.
  In this article, we extend this result to $\Delta$-filtered algebras and $\oDelta$-filtered algebras.  
\end{abstract} 

\section{Introduction}\label{intro}
  Quasi-hereditary algebras were introduced by Cline, Parshall and Scott to study the highest weight categories in Lie theory \cite{CPS}.
  So far, many results have been obtained for quasi-hereditary algebras.
  For example, quasi-hereditary algebras have finite global dimensions, and every algebra is isomorphic to the endomorphism ring of a projective module over a quasi-hereditary algebra.      
  On the other hand, bocs theory was introduced in the context of Drozd's tame and wild dichotomy theorem and Crawley-Boevey applied it to analyze the module categories over tame algebras \cite{C-B}.
  The module categories over bocses behave differently from those over algebras.
  Koenig, K\"ulshammer and Ovsienko connected these theories by giving equivalences between the categories of modules over directed bocses and those of $\Delta$-filtered modules over quasi-hereditary algebras.

  In this article, we would like to extend their results to $\Delta$-filtered algebras and $\oDelta$-filtered algebras. 
  To do so, from Sections \ref{notation} to \ref{tw mod}, we review several facts.
  In \Cref{notation}, we define quasi-hereditary, $\Delta$-filtered (or standardly stratified algebras), and $\oDelta$-filtered algebras.
  In \Cref{bocs}, we define bocses and the categories of modules over them.
  In \Cref{A-inf cat}, we introduce $A_\infty$-categories and, according to \cite{Kel}, show that every $\Ext$-algebra has a structure of an $A_\infty$-category.
  In \Cref{tw mod}, we define twisted modules whose category gives an equivalence between $\F(\oDelta)$ and $\module \B$ for the category $\F(\oDelta)$ of $\oDelta$-filtered modules over a $\oDelta$-filtered algebra and a one-cyclic directed bocs $\B$.
  Now we recall the result of \cite{KKO} (we call this KKO theory), as our main interest of this article, which describes the relation between the quasi-hereditary algebras and directed bocses.
  Their main result is as follows.
        
  \begin{thm}[\cite{KKO} Theorem 1.1, Corollary 1.3, \cite{BKK} Theorem 3.13]\label{KKO main}
    We have a bijection
    \begin{gather*}
      \{\text{Morita equivalence classes of quasi-hereditary algebras}\}\ \\
      \updownarrow \\
      \{\text{Equivalence classes of the module categories over directed bocses}\}.
    \end{gather*}
    Let a quasi-hereditary algebra $A$ and a directed bocs $\B=(B,W)$ correspond via the above bijection.
    Then the right Burt-Butler algebra $R_{\B}$ of $\B$ is Morita equivalent to $A$.
    Moreover, $R_{\B}$ has a homological exact Borel subalgebra $B$.
  \end{thm}

  \Cref{pfa vs pdb} gives a generalization of \Cref{KKO main} to $\oDelta$-filtered algebras and our theorem is as follows.
  
  \begin{thm}[\Cref{main2}]
    We have a bijection
    \begin{gather*}
      \{\text{Morita equivalence classes of $\oDelta$-filtered algebras}\}\ \\
      \updownarrow \\
      \{\text{Equivalence classes of the module categories over one-cyclic directed bocses}\}.
    \end{gather*}
    Let a $\oDelta$-filtered algebra $A$ and a one-cyclic directed bocs $\B=(B,W)$ correspond via the above bijection.
    Then the right Burt-Butler algebra $R_{\B}$ of $\B$ is Morita equivalent to $A$.
    Moreover, $R_{\B}$ has a homological exact Borel subalgebra $B$.
  \end{thm}

  When we generalize \Cref{KKO main} in \Cref{main2}, we face three problems.   
  The first problem concerns with the dimension of the $\Ext$-algebra.
  In \cite{KKO}, the directed bocs is constructed by using the $\Ext$-algebra of standard modules over a quasi-hereditary algebra.
  But in general, the $\Ext$-algebra of properly standard modules over a $\oDelta$-filtered algebra is not finite dimensional.
  To avoid infinite dimensional algebras, we will use a finite dimensional subspace of the $\Ext$-algebra.
  In \Cref{odb form pfa}, we show that the method for construction of bocses by using the subspaces is the generalization of the one used in \cite{KKO}.
  The second problem is on the dimension of $B$ of the bocs $\B=(B,W)$ induced from a $\oDelta$-filtered algebra.
  Since the Gabriel quiver of $B$ has loops but no cycles of length more than $1$, it suffices to show that each $e_iBe_i$ is finite dimensional, which is of course equivalent to the fact that $B$ is so.
  And they are argued in \Cref{odb form pfa}.
  The third problem occurs in the discussion related to \cite{DR} Theorem 2.
  It shows the relationship between Morita equivalence classes of quasi-hereditary algebras and the categories $\F(\Delta)$ of $\Delta$-filtered modules over quasi-hereditary algebras and it is proved by using the fact that an abelian category with some objects $\Theta$ realizes a quasi-hereditary algebra as an endomorphism algebra of $\Ext$-projective objects of $\F(\Theta)$.
  But we can not similarly show the results on $\oDelta$-filtered algebras in place of those for quasi-hereditary algebras, because in general, properly standard modules over $\oDelta$-filtered algebras has self-extensions.
  Hence we can not construct $\Ext$-projective objects of $\F(\oDelta)$ by a way similar to that in \cite{DR}.
  Thus, we use \cite{ADL} Theorem 2.3, which claimed that for any module category over a finite dimensional algebra $C$, there exists a $\oDelta$-filtered algebra $A$ such that $\F(\oDelta_C) \simeq \F(\oDelta_A)$.
  And this is discussed in \Cref{pfa form pdb}.

  Finally, we remark that a generalization of \Cref{KKO main} to $\Delta$-filtered algebras is obtained by arguments which are almost the same as those given in \cite{KKO}.
  The precise statement is as follows.
    
  \begin{thm}[see \cite{BPS} Theorem 12.9]
    We have a bijection
    \begin{gather*}
      \{\text{Morita equivalence classes of $\Delta$-filtered algebras}\}\ \\
      \updownarrow \\
      \{\text{Equivalence classes of the module categories over weakly directed bocses}\}.
    \end{gather*}
    Let a $\Delta$-filtered algebra $A$ and a weakly directed bocs $\B=(B,W)$ correspond via the above bijection.
    Then the right Burt-Butler algebra $R_{\B}$ of $\B$ is Morita equivalent to $A$.
    Moreover, $R_{\B}$ has a homological exact Borel subalgebra $B$.
  \end{thm}

  Note that $\Delta$-filtered algebras and $\oDelta$-filtered algebras are quite natural generalizations of quasi-hereditary algebras, and it is concluded that our main results in this article (\Cref{main2} above) and \cite{BPS} theorem 12.9 generalize Theorem 1.1 to those algebras.


  {\bfseries Acknowledgment}\\
I would like to thank prof. Katsuhiro Uno for many helpful discussions with him and improving this article.

\section{Notations and definitions}\label{notation}

  Throughout this article, let $K$ be an algebraically closed field and $A$ a finite dimensional $K$-algebra with $n$ simple modules (up to isomorphisms).
  The category of a finitely generated left $A$-modules will be denoted by $\module A$ and call its objects just $A$-modules.
  We denote simple $A$-modules by $S_A(i)$ for $1\leq i\leq n$ and corresponding projective indecomposable $A$-modules by $P_A(i)$. 
  But when there is not much danger of confusion, we also write $S(i),P(i)$ for $S_A(i),P_A(i)$, respectively.
  We write $D$ to mean the standard $K$-dual $\Hom_K(-,K)$. 
  For an $A$-module $X$, we write the Jordan-H\"older multiplicity of $S(i)$ in $X$ by $[X:S(i)]$.

  We assume that all categories are $K$-categories.
  In this article, let $\LL$ be the trivial category over $\{1,\dots,n\}$.
  Let $\Cat$ be a category and $X$ and $Y$ objects of $\Cat$.
  We often write $\Cat(X,Y)$ to mean $\Hom_{\Cat}(X,Y)$.
  The {\bfseries dual category} $\D \Cat$ of $\Cat$ is defined as follows.
  Its objects coincide with those of $\Cat$ and $\D \Cat(X,Y) \coloneqq D(\Cat(X,Y))$ where the right hand side is given by the $K$-dual as a vector space.
  A category $\Cat$ is called {\bfseries $\Z$-graded} if every space of morphisms $\Cat(X,Y)$ have a decomposition $\Cat(X,Y)=\bigoplus_{k\in\Z} \Cat(X,Y)_k$, and for $f\in\Cat(X,Y)_k$ and $g\in\Cat(Y,Z)_l$, their composition $gf$ is in $\Cat(X,Z)_{k+l}$.
  For an element $f\in\Cat(X,Y)_k$, the label $k\in\Z$ is called {\bfseries degree} of $f$ and we write by $|f|$ the degree of $f$.
  Let $\A_1,\dots,\A_m$ be $\Z$-graded categories over $\{1,\dots,n\}$.
  Then their {\bfseries tensor product} $\A_m\tensor_{\LL}\cdots\tensor_{\LL}\A_1$ over $\LL$ is the category whose objects are also $1 ,\dots, n$, and
  \[
    \displaystyle
    (\A_m\tensor_{\LL}\cdots\tensor_{\LL}\A_1)(i_0,i_m)_k 
    \coloneqq 
    \bigoplus_{\substack{1 \leq i_1 \dots i_{m-1} \leq n\\k=k_1+\dots+k_m}}\A_m(i_{m-1},i_m)_{k_1}\tensor_K\cdots\tensor_K\A_1(i_0,i_1)_{k_m}.
  \]
  In particular, when $\A=\A_1=\dots=\A_m$, we write $\A^{\tensor m}=\A_m\tensor_{\LL}\cdots\tensor_{\LL}\A_1$.
  Moreover, define the {\bfseries tensor category} $\LL[\A]$ of $\A$ over $\LL$ as follows.
  Its objects are also $1,\dots,n$, and the $K$-vector space of morphisms is defined by $\LL[\A](i,j)=\bigoplus_{r\geq0} \A^{\otimes r}(i,j)$ with $\A^{\tensor 0}=\LL$.

  Let
  \[
    X=(\xymatrix@C=16pt{\cdots \ar[r] & X_{l+1} \ar[r]^-{\partial^X_l} & X_l \ar[r] & \cdots}),\ \ 
    Y=(\xymatrix@C=16pt{\cdots \ar[r] & Y_{l+1} \ar[r]^-{\partial^Y_l} & Y_l \ar[r] & \cdots})
  \]
  be complexes of $A$-modules.
  For any $k \in \Z$, denote by $Y[k]$ the complex $Y$ shifted by $k$ degrees to the left: $Y[k]_l = Y_{l-k}$ and $\partial^{Y[k]}=(-1)^k \partial^Y$.
  Then we write $\Hom_A(X,Y[k])$ as the set of collections $(f^{(l)}:X_l \to Y_{l-k})_{l \in \Z}$ of $A$-homomorphisms i.e.
  \[\xymatrix@C=36pt{
    \cdots \ar[r] & X_{l+1} \ar[r]^-{\partial^X_l} \ar[d]^-{f^{(l+1)}} & X_l \ar[r] \ar[d]^-{f^{(l)}} & \cdots\\
    \cdots \ar[r] & Y_{l-k+1} \ar[r]^-{(-1)^k \partial^Y_{l-k}} & Y_{l-k} \ar[r] & \cdots
  .}\]
  Moreover, if such $f=(f^{(l)})$ is compatible with the differentials of $X$ and $Y$, that is $f^{(l)} \partial^X_l = (-1)^k  \partial^Y_{l-k} f^{(l+1)}$ for any $l \in \Z$, call this $f$ a {\bfseries chain map}.
  And if a map $f \in \Hom_A(X,Y[k])$ is not a chain map, it is called a {\bfseries non-chian map}. 

  Now, we will recall the definitions of some classes of algebras and important modules over them. 

  \begin{df}
  \begin{enumerate}[$(1)$]
    \item For each $i\in\{1,\dots,n\}$, the $A$-module $\Delta_A(i)$, or just $\Delta(i)$, called the {\bfseries standard module}, is defined by the maximal factor module of $P(i)$ having only composition factors $S(j)$ with $j \leq i$.
    Moreover the maximal factor module $\oDelta(i)$ of $\Delta(i)$ such that $[\oDelta(i):S(i)]=1$ is called the {\bfseries properly standard module}.
    Let $\Delta=\{\Delta(1),\dots,\Delta(n)\}$ and $\oDelta=\{\oDelta(1),\dots,\oDelta(n)\}$, and we call them a standard system and a properly standard system, respectively. 
    \item We say that a module $X$ has a {\bfseries $\Delta$-filtration}, or $X$ is {\bfseries $\Delta$-filtered}, if there is a submodule sequence $0=X_m\subset\dots \subset X_1 \subset X_0=X$ such that for each $1 \leq k \leq m$, $X_{k-1}/X_k\cong \Delta(j)$ for some $j\in\{1,\dots,n\}$.
    Write $\F(\Delta)$ to mean the full subcategory of $\module A$ whose objects are modules with $\Delta$-filtrations.  
    Similarly we define $\oDelta$-filtered modules and the category $\F(\oDelta)$.
    \item A pair of an algebra and a total order $(A,\leq)$, or just $A$, is called a {\bfseries $\Delta$-filtered algebra} (resp. a {\bfseries $\oDelta$-filtered algebra}) provided that every $P(i)$ has a $\{\Delta(i),\dots,\Delta(n)\}$-filtration (resp. a $\{\oDelta(i),\dots,\oDelta(n)\}$-filtration).
    \item If $(A,\leq)$ is a $\Delta$-filtered algebra and $\Delta=\oDelta$, then it is called a {\bfseries quasi-hereditary algebra}.
  \end{enumerate}
  \end{df}

  \begin{rmk}\label{standard modules}
    \begin{enumerate}[(1)]
      \item A $\Delta$-filtered algebra is also called a standardly stratified algebras.
      \item The standard modules and properly standard modules are also defined by the following:
      \[\Delta(i)=\sum_{\substack{j>i \\ \varphi\in\Hom_A(P(i),P(j))}}\image\varphi,\ \ \ \ \oDelta(i)=\sum_{\substack{j\geq i \\ \varphi\in\rad_A(P(i),P(j))}}\image\varphi.\]
    \end{enumerate}
  \end{rmk}
  
  

  Hereafter, let $e_1,\dots,e_n$ be pairwise orthogonal basic primitive idempotents of an algebra $B$ and write $Be_i$ by $P_B(i)$.

  \begin{df}
    \begin{enumerate}[$(1)$]
      \item Let $A$ be a $\Delta$-filtered algebra.
            A subalgebra $B$ of $A$ is called an {\bfseries exact Borel subalgebra} if the following conditions hold.
            \begin{itemize}
              \item $B$ is directed, i.e., $\rad_B(P_B(i),P_B(j))=0$ for $i \leq j$.
              \item $A \tensor_B -: \module B \to \module A$ is exact.
              \item $\Delta_A(i) \cong A \tensor_B S_B(i)$.
            \end{itemize}
      \item Let $A$ be a $\oDelta$-filtered algebra.
            A subalgebra $B$ of $A$ is called a {\bfseries proper Borel subalgebra} if the following conditions hold.
            \begin{itemize}
              \item $B$ is one-cyclic directed, i.e., $\rad_B(P_B(i),P_B(j))=0$ for $i < j$.
              \item $A \tensor_B -: \module B \to \module A$ is exact.
              \item $\oDelta_A(i) \cong A \tensor_B S_B(i)$.
            \end{itemize}
      \item A subalgebra $B$ of $A$ is {\bfseries homological} if for any $B$-modules $X,Y$, natural maps
            \[\Ext^k_B(X,Y) \to \Ext^k_A(A \tensor_B X,A \tensor_B Y)\]
            are epimorphisms for $k\geq1$ and isomorphisms for $k\geq2$.
    \end{enumerate}
  \end{df}
  
  


\section{Bocses}\label{bocs}
  The bocs theory was introduced by Drozd's tame and wild dichotomy theorem, and Crawley-Boevey studied bocses in \cite{C-B}.

  \begin{df}
    A bocs is $\B=(B,W,\varepsilon,\mu)$, or just $(B,W)$, consisting of a basic $K$-algebra $B$ and a $B$-bimodule $W$ which has a $B$-coalgebra structure, that is, there exist a $B$-bilinear counit $\varepsilon:W \to B$ and a $B$-bilinear comultiplication $\mu:W \ni w \mapsto \sum w_1 \otimes w_2 \in W \otimes_B W$ (using sigma notation), and the following diagrams are commutative:
    \[\xymatrix{W\ar[r]^-\mu \ar[d]_-\mu & W \otimes W\ar[d]^-{\id_W\otimes\mu} & & B \otimes W & W \otimes W\ar[r]^-{\id_W\otimes\varepsilon}\ar[l]_-{\varepsilon\otimes \id_W} & W \otimes B\\
    W \otimes W\ar[r]_-{\mu\otimes \id_W} & W \otimes W \otimes W &&& W\ar[u]^-\mu \ar[ru]_-{r_W^{-1}} \ar[lu]^-{l_W^{-1}}}\]
    with isomorphisms $l_W:B\otimes W\ni b\otimes w \mapsto bw\in W$ and $r_W:W\otimes B\ni w\otimes b \mapsto wb\in W$.
  \end{df}

  We will always assume that the algebra $B$ and $B$-bimodule $W$ of a bocs are finite dimensional and the counit $\varepsilon$ is surjective. 

  \begin{df}
  \begin{enumerate}[$(1)$]
    \item A bocs $\B=(B,W)$ is said to have a projective kernel if $\overline{W} \coloneqq \kernel \varepsilon$ is a projective $B$-bimodule. 
    \item A bocs $\B=(B,W)$ with a projective kernel is called
    \[
    \begin{cases*}
      directed & if $\rad_B(P_B(i),P_B(j))=0$ for $i \leq j$ and $\overline{W} \cong \oplus_{i>j}(Be_i \otimes_K e_jB)^{d_{ij}}$,\\
      \text{weakly directed} & if $\rad_B(P_B(i),P_B(j))=0$ for $i \leq j$ and $\overline{W} \cong \oplus_{i \geq j}(Be_i \otimes_K e_jB)^{d_{ij}}$,\\
      \text{one-cyclic directed} & if $\rad_B(P_B(i),P_B(j))=0$ for $i<j$ and $\overline{W} \cong \oplus_{i>j}(Be_i \otimes_K e_jB)^{d_{ij}}$,
    \end{cases*}
    \]
    for some $d_{ij}\geq0$.
    
  \end{enumerate}
  \end{df}

  \begin{df}
  The category $\module \B$ of finite dimensional modules over a bocs $\B=(B,W)$ is defined as follows:
    \begin{description}
    \item[objects] finite dimensional left $B$-modules
    \item[morphisms] for $B$-modules $X$ and $Y$, define $\Hom_\B(X,Y) = \Hom_B(W \tensor_B X,Y)$
    \begin{description}
    \item[composition] for $B$-modules $X,Y,Z$ and morphisms $f\in\Hom_\B(X,Y),g\in\Hom_\B(Y,Z)$, the composition $gf$ of $f$ and $g$ is given by:
      \[\xymatrix{W \tensor X \ar[r]^-{\mu \tensor \id_X} & W \tensor W \tensor X \ar[r]^-{\id_W \tensor f} & W \tensor Y \ar[r]^-g & Z.}\]
      \item[unites] the unite morphism $\id_X^{\B} \in \End_\B(X)$ of a $B$-module $X$ is given by the composition of the following maps:
      \[\xymatrix{W \tensor X \ar[r]^-{\varepsilon \tensor \id_X} & B \tensor X \ar[r]^-{l_X} & X.}\]
    \end{description}
  \end{description}
  \end{df}

  \begin{rmk}
    Since a standard adjunction gives the canonical isomorphism
    \[\Hom_B(W \tensor_B X,Y) \cong \Hom_{B \tensor B^{\op}}(W,\Hom_K(X,Y)),\]
    we may also define $\Hom_\B(X,Y)$ by $\Hom_{B \tensor B^{\op}}(W,\Hom_K(X,Y))$.
    In this case, the composition and unites in $\module\B$ are as follows.
    \begin{description}
      \item[composition] for $f\in\Hom_\B(X,Y)$ and $g\in\Hom_\B(Y,Z)$, the composition $gf$ is given by:
      \[\xymatrix{W \ar[r]^-{\mu} & W \tensor_B W \ar[r]^-{g \tensor f} & \Hom_K(Y,Z) \tensor_B \Hom_K(X,Y) \ar[r]^-{\circ} & \Hom_K(X,Z),}\]
      where $\circ$ is the usual composition of maps.
      \item[unites] the unite morphism $\id_X^{\B} \in \End_\B(X)$ is given by the composition of the following maps:
      \[\xymatrix{W \ar[r]^-{\varepsilon} & B \ar[r]^-{\gamma_X} & \End_K(X),}\]
      where $\gamma_X(1_B)=\id_{X}$.
    \end{description}
  \end{rmk}

  \begin{df}[\cite{BB}]
    Let $\B=(B,W)$ be a bocs.
    The right Burt-Butler algebra $R=R_{\B}$ of the bocs $\B$ is defined by $\End_{\B}(B)^{\op}$ and whose multiplication is the composition of morphisms in $\module\B$ with $1_R=\id_B^{\B}$.
  \end{df}
  Similarly we can also define the left Burt-Butler algebra $L_{\B}=\End_{\B^{\op}}(B)$ of $\B$.
  Now we give some facts for right Burt-Butler algebras which are necessary for our argument.

  \begin{lem}[\cite{BB}]\label{BB}
  Let $\B$ be a bocs with projective kernel and $R$ the right Burt-Butler algebra of $\B$.
  Then the following conditions hold.
  \begin{enumerate}[$(1)$]
    \item\label{BB Thm 2.5} There is an equivalence between categories
          \[R\tensor_B-:\module \B \to {\rm Ind}(B,R),\]
    where ${\rm Ind}(B,R)$ is the full subcategory of $\module R$ consisting of the modules $R \otimes_B X$ for $B$-modules $X$.
    \item\label{BB Thm 3.8} There are surjective maps
          \[\Ext^k_B(X,Y) \to \Ext^k_R(R \tensor X,R \tensor Y) \text{ for } k\geq1,\]
          and they are bijective for $k\geq2$.
    \item\label{BB Sec 2} The algebra $R$ is a projective $B$-module.
  \end{enumerate}
  \end{lem}

\section{$A_\infty$-categories}\label{A-inf cat}
  In this section, we will recall the definition of $A_\infty$-categories.
  As an important example, $\Ext$-algebras have an $A_\infty$-structure, see \Cref{(4.3)}, and we will use this fact in \Cref{odb form pfa} to construct bocses from $\oDelta$-filtered algebras.
  \begin{df}
    A $\Z$-graded $K$-category $\A$ is called an {\bfseries $A_\infty$-category} if there are $K$-bilinear functors $m_r:\A^{\otimes r} \to \A$ for $k \geq 1$ such that each functor $m_r$ of degree $2-k$ satisfies the condition
    \[\sum_{\substack{k=r+t+u \\ r,t,u\geq0}}(-1)^{r+tu}m_{r+1+u}(\id^{\otimes r} \otimes m_t \otimes \id^{\otimes u})=0, \text{ for any } k \geq 1.\]
    We often write the $A_\infty$-category by a pair $(\A,m)$.
  \end{df}
  
  
  \begin{lem}[\cite{Kel} (3.6) Lemma]\label{(3.6) Lem}
    Let $\A$ be a $\Z$-graded category and $m_r:\A^{\otimes r} \to \A$ for $ r \geq 1$ graded functors of degree $2-r$.
    We define the suspension $s\A$ of $\A$ given by $(s\A)_k=\A_{k+1}$ and consider the family of graded functors $b_r:s\A^{\otimes r} \to s\A$ of degree $1$ which make the following diagram commutative   
    \[
      \xymatrix{
        (s\A)^{\otimes r} \ar[r]^-{b_r} & s\A \\
        \A^{\otimes r} \ar[r]_-{m_r} \ar[u]^-{s^{\otimes r}} & \A \ar[u]_-s
      }
    \]
  where $s:\A \to s\A$ is the canonical functor of degree $-1$.
  Then the following are equivalent
    \begin{enumerate}[$(1)$]
    \item $(\A,m)$ is an $A_\infty$-category .
    \item For any $k \geq 1$, we have 
      \[\sum_{\substack{k=r+t+u \\ r,u\geq0,t\geq1}}b_{r+1+u}(\id^{\otimes r} \otimes b_t \otimes \id^{\otimes u})=0.\]
    \item A differential coalgebra $\overline{T}s\A=\oplus_{r \geq 1} (s\A)^{\otimes r}$ with the comultiplication 
      \[\mu:\overline{T}s\A \to \overline{T}s\A \tensor \overline{T}s\A ;\ a_r \tensor \dots \tensor a_1 \mapsto \sum_{t=1}^{r-1} (a_r \tensor \dots \tensor a_{t+1}) \tensor (a_t \tensor \dots \tensor a_1),\]
      gives $b:\overline{T}s\A \to \overline{T}s\A$ of degree $1$ such that $b^2=0$ and $\mu b = (1 \tensor b + b \tensor 1) \mu$.
    \end{enumerate}
  \end{lem} 
  In particular, the equivalence between $(2)$ and $(3)$ is given by defining the composition of $b$ that maps $(s\A)^{\otimes k}$ to $(s\A)^{\otimes l}$ as $\sum_{\substack{k=r+t+u \\ r,u\geq0,t\geq1}}(\id^{\otimes r} \otimes b_t \otimes \id^{\otimes u})$, where $l=r+1+u$.

  \begin{prop}[\cite{Kel} (3.1), \cite{KKO} Theorem 4.3]\label{Thm 4.3}
    Let $(\A,\lambda)$ be an differential graded category.
    Then the homology category $\HH^*(\A)$ of $\A$, where the homology is taken with respect to $\lambda_1$, is also an $A_\infty$-category $(\HH^*(\A),m)$ with $m_1=0$ and $m_2$ being induced from $\lambda_2$.
  \end{prop}
  
  \begin{eg}[\cite{KKO} (4.3)]\label{(4.3)}
    Let $X$ be an $A$-module and $\mbP_X$ a projective resolution of $X$ with differential $\partial_X$.
    Let $\A$ be a category with one object $X$ and $\A(X,X)=\End_A(\mbP_X)$ the endomorphism ring of $\mbP_X$ as a complex of $A$-modules. 
    Then the $K$-linear map $\lambda_1:\A \to \A$ with $f \mapsto \partial_X f - (-1)^k f \partial_X$ for homogeneous $f$ of degree $k$, the natural composition $\lambda_2: \A^{\otimes2} \to \A$, and $\lambda_r=0$ for $r\geq3$ make $(\A,\lambda)$ a differential graded category.
    Moreover the morphism space $\HH^*(\A)(X,X)$ of the homology category $\HH^*(\A)$ of $\A$ with respect to $\lambda_1$ is just the $\Ext$-algebra $\Ext^*_A(X,X)=\oplus_{k\geq0}\Ext^k_A(X,X)$ of $X$ and is an $A_\infty$-category with $m_1=0$ and $m_2$ being the Yoneda product.
  \end{eg}

\section{Twisted modules}\label{tw mod}
  In this section, we define twisted modules in order to prepare for the proof of \Cref{thm1}.
  For detail, refer to \cite{Kel} or \cite{KKO}.
  
  Let $(\A,m)$ be an $A_\infty$-category.
  First, we define the category $\add\A$.
  Its objects are $\LL$-modules, i.e., functors $X:\LL \to \module K$, and for $\LL$-modules $X$ and $Y$, we have
  \[\add\A(X,Y)_k \coloneqq \bigoplus_{1\leq i,j\leq n} \Hom_K(X(i),Y(j)) \tensor \A(i,j)_k.\]
  Moreover we define the graded multiplications $b_r^a$ of $\add s\A$ as follows:
  \[b_r^a((f_r \tensor sa_r) \tensor \cdots \tensor (f_1 \tensor sa_1)) \coloneqq -f_r \circ \dots \circ f_1 \tensor b_r(sa_r \tensor \cdots \tensor sa_1).\]
  Then $\add s\A$ satisfies the condition (2) in \Cref{(3.6) Lem}, which was shown in \cite{KKO} Lemma 5.1, and hence $\add \A$ has an $A_\infty$-structure.

  Next we define twisted modules.
  A {\bfseries pretwisted module} $(X,\delta)$ is a pair of an $\LL$-module $X$ and $\delta=\sum_r (f_r \tensor a_r) \in \add\A(X,X)_1$.
  A {\bfseries pretwisted submodule} $(X',\delta')$ of $(X,\delta)$ is defined by the following:
  A family of subspaces $X'(i) \subset X(i)$ such that the restriction $f_r|_{X'(i)}$ of $f_r:X(i) \to X(j)$ to $X'(i)$ maps into $X'(j)$.
  Moreover $\delta'$ is given by $\sum_r (f_r|_{X'} \tensor a_r)$.
  There is an obvious way to define the notion of a {\bfseries pretwisted factor module} $(X/X',\delta/\delta')$.
  Finally, a pretwisted module $(X,\delta)$ is called a {\bfseries twisted module} if the following conditions hold:  
  \begin{enumerate}[$(1)$]
    \item There is a chain of pretwisted submodules 
          \[(0,0) = (X_N,\delta_N) \subset  \dots \subset (X_1,\delta_1)  \subset (X_0,\delta_0) = (X,\delta)\]
          of $(X,\delta)$ such that for each $1 \leq i \leq N$, $(X_{i-1},\delta_{i-1})/(X_i,\delta_i)=(X_{i-1}/X_i,0)$.
    \item $\sum_{r\geq1} (-1)^{r(r-1)/2} m_r(\delta, \dots, \delta)=0$.
  \end{enumerate}
  The first condition is called {\bfseries triangularity} and the second one is called the {\bfseries Maurer-Cartan equation}.
  Moreover for any morphism $X \to X'$ in ${\add\A}$, we define a morphism of twisted modules $(X,\delta)\to(X',\delta')$, 
  and we will denote $\twmod\A$ the category of twisted modules.
  Further we give the graded multiplications $b_r^{tw}$ of $\twmod s\A$ by
  \[b_r^{tw}(st_r \tensor \cdots \tensor st_1) \coloneqq \sum_{i_0 , \dots , i_r \geq 0} b_{i_0 + \dots + i_r + r}^a(s\delta_{X_r}^{\tensor i_r} \tensor st_r \tensor s\delta_{X_{r-1}}^{\tensor i_{r-1}} \tensor \cdots \tensor st_1 \tensor s\delta_{X_0}^{\tensor i_0}).\]
  Then $\twmod s\A$ also satisfies the condition (2) in \Cref{(3.6) Lem}, see \cite{KKO} Lemma 5.3.
  Thus $\twmod\A$ is an $A_\infty$-category.

  In order to prove \Cref{thm1}, we mention the next lemma without the poof.
  \begin{lem}[\cite{Kel} (7.7), \cite{KKO} Theorem 5.4]\label{Thm 5.4}
    For an $n$-tuple of $A$-modules $X=(X_1,\dots,X_n)$, we have $\F(X) \simeq \HH^0(\twmod \Ext^*_A(X,X))$.
  \end{lem}
    
  Next we consider the category $\conv\A$ whose objects are $\LL$-modules and 
  \[
    \conv\A(X,Y)_k \coloneqq \bigoplus_{1\leq i,j\leq n} \Hom_K(\D \A(i,j)_k,\Hom_K(X(i),Y(j))).
  \]
  Since there is a natural isomorphism of vector spaces
  \[
    M_{V,W}: V \tensor W \to \Hom(DW, V);\ M_{V,W}(v \tensor w)(\chi) = \chi(w)v \text{ for } \chi \in DW
  \]
  by regarding $V=s\A(i,j)$ and $W=\Hom_K(X(i),Y(j))$, the functor $M$ is an equivalence between $\add s\A$ and $\conv s\A$. 
  Next, we construct a functor $b^c$ for $s(\conv\A)$ satisfying (2) in \Cref{(3.6) Lem}.
  The functor $b^c_r:s(\conv\A)^{\otimes r} \to s(\conv\A)$ is defined by $sF_r \tensor \cdots \tensor sF_1 \mapsto v_r [sF_r \tensor \cdots \tensor sF_1] d_r$,
  \[\xymatrix{
    \displaystyle\bigoplus_{1\leq i_0,\dots,i_r\leq n}\bigotimes_{1 \leq t \leq r} \D s\A(i_{t-1},i_t) \ar[rr]^-{ sF_r \tensor \cdots \tensor sF_1} & \ar@{=>}[d]_-{b^c_r} & \displaystyle\bigoplus_{1\leq i_0,\dots,i_r\leq n}\bigotimes_{1 \leq t \leq r} \Hom_K(X_{t-1}(i_{t-1}),X_t(i_t)) \ar[d]^{v_r}\\
    \D s\A(i_0,i_r) \ar[rr] \ar[u]^-{d_r} && \Hom_K(X_0(i_0),X_r(i_r)),  
  }\]
  where the three functors $d_r$, $[sF_r \tensor \cdots \tensor sF_1]$, and $v_r$ are defined as follows.
  \begin{itemize}
    \item $d_r(\chi)(sa_r \tensor \cdots \tensor sa_1)=\chi(b_r(sa_r \tensor \cdots \tensor sa_1))$,
    \item $[sF_r \tensor \cdots \tensor sF_1] (\chi_r \tensor \cdots \tensor \chi_1) = {\rm sgn}(sF,\chi) (sF_r(\chi_r) \tensor \cdots \tensor sF_1(\chi_1))$\\
    with ${\rm sgn}(sF,\chi)=(-1)^{\sum_{1\leq t < u\leq r}|sF_t||\chi_u|+1}$,
    \item $v_r (f_r \tensor \cdots \tensor f_1) = f_r \circ \dots \circ f_1.$
  \end{itemize}
  Then $b^c$ satisfies the condition (2) in \Cref{(3.6) Lem} for $\conv s\A$, and we have $Mb_r^a=b_r^cM^{\otimes r}$ (\cite{KKO} Lemma 6.1).
  Finally, we consider objects corresponding to twisted modules.
  Let $\QQ=\D s\A$.
  Then the next lemma and proposition hold.
  \begin{lem}[\cite{KKO} Lemma 6.2]\label{Lem 6.2}
    There is a one-to-one correspondence between $\LL[\QQ_0]$-modules $\mathfrak{X}_{\delta}$ and pairs $(X,M(s\delta))$ where $(X,\delta)$ are pretwisted modules over $\A$.
    This relation is given by $\mathfrak{X}_{\delta}=\sum_{r\geq1} v_r \circ (M(s\delta))^{\otimes r}$ as an $\LL[\QQ_0]$-module.
  \end{lem}
  
  \begin{prop}[\cite{KKO} Proposition 6.3]\label{Prop 6.3}
    Let $(X,\delta)$ be a pretwisted module over $\A$ and $\mathfrak{X}_{\delta}$ an $\LL[\QQ_0]$-module given in \Cref{Lem 6.2}.
    Put $\delta=\sum_r(f_r \tensor a_r)$.
    Then the following hold:
      \begin{enumerate}[$(1)$]
      \item For $\chi \in \QQ_0(i,j), x \in X$, we have $\mathfrak{X}_{\delta}(\chi)(x) = \sum_k \chi(sa_r)f_r(x)$.
      \item $\delta=0$ if and only if $\mathfrak{X}_{\delta}$ is semi-simple.
      \item $(X',\delta')$ is a pretwisted submodule of $(X,\delta)$ if and only if $\mathfrak{X}_{\delta'}$ is a submodule of $\mathfrak{X}_{\delta}$.
      \item If $(X',\delta') \subset (X,\delta)$, then $\mathfrak{X}_{\delta/\delta'}=\mathfrak{X}_{\delta}/\mathfrak{X}_{\delta'}$.
      \item $(X,\delta)$ satisfies the triangularity condition if and only if for some $r\geq1$, $v_r \circ (M(s\delta))^{\otimes r}:(\QQ_0)^{\otimes r} \to \End_A(X)$ is the zero map.
      \item $(X,\delta)$ satisfies Maurer-Cartan equation if and only if $\mathfrak{X}_{\delta} \circ d=0$ with $d=(d_r)_{r \geq 1}$ if and only if $\mathfrak{X}_{\delta}$ is a $B=(\LL[\QQ_0]/(\LL[\QQ_0]\cap\image d))$-module, where the functors $d_r:\D s\A \to \D s\A^{\otimes r}$ are defined by 
      \[
        d_r(\chi)(sa_r \tensor \cdots \tensor sa_1)
        =\chi(b_r(sa_r \tensor \cdots \tensor sa_1))
        =\chi (s m_r(a_r \tensor \cdots \tensor a_1)).
      \]
    \end{enumerate}
  \end{prop}
  This proposition will be used in \Cref{thm1} to show the equivalence between $\module \B$ and $\HH^0(\twmod \HH^*(\A))$.

\section{$\oDelta$-filtered algebras versus one-cyclic directed bocses}\label{pfa vs pdb}
  In this section, let $A$ be a $\oDelta$-filtered algebra, unless otherwise noted.
  We will show the relationship between $\oDelta$-filtered algebras and one-cyclic directed bocses.
  To do this, we imitate the arguments in \cite{KKO}.
  In the case of a $\oDelta$-filtered algebra, we must face three problems.
  They come from the fact that properly standard modules may have self-extensions.

  \subsection{One-cyclic directed bocses from $\oDelta$-filtered algebras}\label{odb form pfa}
    In this subsection, we want to construct a one-cyclic bocs $\B_A$ from a $\oDelta$-filtered algebra $A$ such that $\module\B_A \simeq \F(\oDelta_A)$.
    Difficulties for generalizing \Cref{KKO main} to $\oDelta$-filtered algebras lie in the following.
    In \cite{KKO}, they used the $\Ext$-algebra $\Ext^*_A(\Delta,\Delta)=\bigoplus_{k \geq 0}\bigoplus_{1\leq i,j \leq n}\Ext^k_A(\Delta(i),\Delta(j))$ of $\Delta$ and it is finite dimensional for a quasi-hereditary algebra $A$.
    But in the case for $A$ being $\oDelta$-filtered algebra, $\Ext_A^*(\oDelta,\oDelta)$ is not finite dimensional, in general.
    This is the first problem.
    Moreover in \cite{KKO}, the differential $b$ of the tensor algebra $T(s\Ext^*_A(\Delta,\Delta))$ can be controlled by finitely many elements since we have $b_r=0$ for $k>n$.
    Obviously, for a $\oDelta$-filtered algebra $A$, the algebra $T(s\Ext^*_A(\oDelta,\oDelta))$ can not be.
    It is the second problem.
    So in order to avoid infinite dimensional algebras, we generalize a method of constructing bocses.

    Let $\mbP_i$ be a projective resolution of $\oDelta(i)$ with differential $\partial$ for any $1 \leq i \leq n$ and $\A$ the $\Z$-graded category with objects $1,\dots,n$ and $\A^k(i,j)=\Hom_A(\mbP_i,\mbP_j[k])$.
    Then $\A$ is a differential graded category with the multiplication being compositions of morphisms and differential $d$ defined by $d(f)= \partial \circ f - (-1)^k f \circ \partial$ for $f \in \A^k$.
    Now let $Z^*(\A)$, $B^*(\A)$, and $\HH^*(\A)$ be the cocycles, coboundaries, and cohomology of $\A$, respectively.
    Then $Z^*(\A)$ is the set of chain morphisms in $\A$ and $B^*(\A)$ is that of null-homotopic morphisms i.e. morphisms homotopic to zero morphisms.
    And we can identify $\HH^*(\A)$ with $\Ext^*_A(\oDelta,\oDelta)$ by \Cref{(4.3)}.
    Further $\HH^*(\A)$ is an $A_\infty$-category with multiplications $m_r$ by \Cref{(4.3)} again.

    At first, we recall that the method of construction of directed bocses from quasi-hereditary algebras in \cite{KKO}.
    We will write $s\HH^k(\A)$ to mean $(s\HH^*(\A))^k=\HH^{k+1}(\A)$.
    \begin{enumerate}[Step 1]
      \item Let $A$ be a quasi-hereditary and $\HH^*(\A)=\Ext_A^*(\Delta,\Delta)$.
      \item Consider the the dual maps $d_k:\QQ \to \QQ^{\otimes r}$ of $b_r$ given in \Cref{(3.6) Lem}, where $\QQ_k(i,j)=D((s\HH^{k}(\A))(i,j))$.
      \item Let $\LL[\QQ]$ be the tensor category of $\QQ$ over $\LL$.
        Then $\LL[\QQ]$ is a differential graded category with differential $d$.
      \item Let $U=\LL[\QQ]/I$, where the ideal $I$ of $\LL[\QQ]$ is generated by $\QQ_{\leq -1}$ and $d(\QQ_{-1})$.
        Since $I$ is a differential ideal with respect to $d$, the factor $U$ is also a differential graded category.
        Moreover, $U$ is freely generated over $B=\LL[\QQ_0]/(\LL[\QQ_0] \cap I)$ by $\QQ_1$.
      \item Put $W=U_1/d(B)$ and take the natural epimorphism $\pi: U_1 \to W$.
        Consider the two homomorphisms $\mu: W \to W \otimes W$ and $\varepsilon: W \to B$ such that the following diagrams commute, respectively,
      \[\xymatrix{
        U_1 \ar[r]^-{d} \ar[d]_-\pi & U_1 \tensor U_1 \ar[d]^-{\pi \tensor \pi} &   U_1 \ar[r]^-{\cong} \ar[d]_-\pi & (\bigoplus_i B \omega_i B) \oplus \overline{U} \ar[d]^-{\tilde{\varepsilon}}\\
        W \ar[r]_-\mu & W \tensor W ,&                                                             W \ar[r]_-{\varepsilon} & B,
      }\]
      where $\omega_i \in \QQ_1(i,i)$ are elements corresponding to $\id_{\oDelta(i)}$, and $\tilde{\varepsilon}$ maps $\omega_i$ to $e_i$ and $\overline{U}$ to zero.
      Then $(B,W,\varepsilon,\mu)$ is a directed bocs.
    \end{enumerate}

    The first problem lies in Steps 1 and 2.
    We notice that for a $\oDelta$-filtered algebra $A$, the $\Ext$-algebra $\HH^*(\A)=\Ext^*_A(\oDelta,\oDelta)$ is infinite dimensional in general.
    It is difficult to take the dual of $s\HH^*(\A)$ with respect to $K$-bases.
    In order to argue Step 2 for finite dimensional spaces, we consider a subspace of of $s\HH^*(\A)$.

    \begin{prop}\label{mult tensor alg}
      \begin{enumerate}
        \item Let $b_r:s\HH^*(\A)^{\otimes r} \to s\HH^*(\A)$ be a graded maps of degree $1$ induced from the $A_\infty$-category $\HH^*(\A)$ by \Cref{(3.6) Lem}.
          Consider the graded linear maps $b'_r:(s\HH^{\leq1}(\A))^{\otimes r} \to s\HH^{\leq1}(\A)$ of degree $1$ defined by
          \[
            b'_r(sa_r,\dots,sa_1)=
            \begin{cases}
              b_r(sa_r , \dots , sa_1) & \text{ for } \sum_{t=1}^r |sa_t|\leq0,\\
              0 & \text{ for } \sum_{t=1}^r |sa_t|\geq1.
            \end{cases}
          \]
          Then for any $k \geq 1$, we have 
          \[\sum_{\substack{k=r+t+u \\ r,u\geq0,t\geq1}}b'_{r+1+t}(\id^{\otimes r} \otimes b'_t \otimes {\rm id}^{\otimes u})\iota^{\otimes k}=0,\]
          where $\iota:s\HH^{\leq0}(\A) \to s\HH^{\leq1}(\A)$ is the canonical injection.
        \item Take the dual statement of the above and let $Q=s\HH^*(\A)$.
          Then we get graded maps $d'_r=D(b'_r):Q_{\geq -1} \to Q_{\geq -1}^{\otimes r}$ of degree $1$ satisfying
          \[\sum_{\substack{k=r+t+u \\ r,u\geq 0, t\geq 1}}p^{\otimes k}(\id^{\otimes r} \otimes d'_t \otimes {\rm id}^{\otimes u})d'_{r+1+u}=0,\]
          for $k \geq 1$, where $p:Q_{\geq -1} \to Q_{\geq 0}$ is the canonical surjection.
        \item Consider the factor category $T(Q_{\geq 0})/d'(Q_{\geq -1})$.
          Then it is a differential graded category with differential $d'$ induced from the maps $d'_r$.
      \end{enumerate}
    \end{prop}
    \begin{proof}
        The claims 1. and 2. can be checked by routines.
        So we prove 3.
        Let $d':T(Q_{\geq -1}) \to T(Q_{\geq -1})$ be the graded map whose component mapping $Q_{\geq -1}^{\otimes l}$ to $Q_{\geq -1}^{\otimes k}$ is defined by 
        \[
          \sum_{\substack{k=r+t+u \\ r,u\geq 0, t\geq 1}}(\id^{\otimes r} \otimes d'_t \otimes \id^{\otimes u}),
        \]
        where $l=r+1+u$.
        Then we have $p \circ d'^2=0$ by 2.
        Moreover $d'$ satisfies $d' m'= m' (1 \otimes d' + d' \otimes 1)$ for the natural multiplication
        \[
          m':T(Q_{\geq -1}) \otimes T(Q_{\geq -1}) \to T(Q_{\geq -1});\ (q_r \otimes \dots \otimes q_{t+1}) \cdot (q_t \otimes \dots \otimes q_1) \mapsto q_r \otimes \dots \otimes q_1
        \]
        of $T(Q_{\geq -1})$.
        Since the ideal $I'$ of $T(Q_{\geq -1})$ generated by $Q_{-1}$ and $d'(Q_{-1})$ is closed under $d'$, the factor $T(Q_{\geq -1})/I' \cong T(Q_{\geq 0})/d'(Q_{-1})$ is a differential graded category.
    \end{proof}

    Hence the algebra $U$ in Step 4 can be replaced with such the differential graded algebra $T(Q_{\geq 0})/d'(Q_{-1})$ with differential $d'$ above.

    Now we consider the construction of a bocs $\B_A=(B,W)$ from a $\oDelta$-filtered algebra $A$ as follows.
    \begin{enumerate}[Step 1]
      \item Let $A$ be a $\oDelta$-filtered algebra and $\HH^*(\A)=\Ext_A^*(\oDelta,\oDelta)$.
      \item Consider the maps $d'_r:\QQ_{\geq-1} \to \QQ_{\geq-1}^{\otimes r}$ given in \Cref{mult tensor alg}, where $\QQ_k(i,j)=\\D((s\HH^{k}(\A))(i,j))$.
      \item Let $\LL[\QQ_{\geq -1}]$ be the tensor category of $\QQ_{\geq -1}$ over $\LL$.
      \item Put $U=\LL[\QQ_{\geq -1}]/I'$, where the ideal $I'$ of $\LL[\QQ_{\geq -1}]$ is generated by $\QQ_{-1}$ and $d'(\QQ_{-1})$.
        Then the factor $U$ is a differential graded category by \Cref{mult tensor alg}.
        Moreover, $U$ is freely generated over $B=\LL[\QQ_0]/d'(\QQ_{-1})$ by $\QQ_1$.
      \item Put $W=U_1/d'(B)$ and take the natural epimorphism $\pi: U_1 \to W$.
        Consider the two homomorphisms $\mu: W \to W \otimes W$ and $\varepsilon: W \to B$ such that the following diagrams commute, respectively,
      \[\xymatrix{
        U_1 \ar[r]^-{d} \ar[d]_-\pi & U_1 \tensor U_1 \ar[d]^-{\pi \tensor \pi} &   U_1 \ar[r]^-{\cong} \ar[d]_-\pi & (\bigoplus_i B \omega_i B) \oplus \overline{U} \ar[d]^-{\tilde{\varepsilon}}\\
        W \ar[r]_-\mu & W \tensor W ,& W \ar[r]_-{\varepsilon} & B,
      }\]
      where $\omega_i \in \QQ_1(i,i)$ are elements corresponding to $\id_{\oDelta(i)}$, and $\tilde{\varepsilon}$ maps $\omega_i$ to $e_i$ and $\overline{U}$ to zero.
      Then $\B_A=(B,W,\varepsilon,\mu)$ is a bocs with projective kernel, see \cite{KKO} Lemmas 7.5 to 7.7.
    \end{enumerate}
    
    We will show that this $\B=(B,W)$ is a one-cyclic direct bocs with $B$ being finite dimensional.
    But when $A$ is $\oDelta$-filtered algebra, \cite{KKO} did not guarantee that $B$ is finite dimensional in Step 4.
    And this is the second problem for the generalization of \Cref{KKO main}.
    Remark that $B$ is finite dimensional if so are $e_iBe_i$ for all $1 \leq i \leq n$, because $B$ is one-cyclic directed by the construction (see \Cref{one-cyclic directed} below).
    In order to show that $e_iBe_i$ are finite dimensional, we prove that $\End_A(\Delta(i))$ and $e_iBe_i$ are Morita equivalent in \Cref{EB Morita equiv} below.
    To show this, we use the following remark and theorem.

    \begin{rmk}[\cite{LPWZ}, \cite{M}]\label{Mer}
      We have a decomposition
      \[\A =Z^*(\A) \oplus L^*(\A) = \HH^*(\A) \oplus B^*(\A) \oplus L^*(\A)\]
      of the graded differential category $\A$ for a subspace $L^*(\A)$ of $\A$ and identify $L^*(\A)$ with $\A/Z^*(\A)$.
      Consider the graded map $G:\A \to \A$ of degree $-1$ satisfying $G|_{L^k(\A)\oplus \HH^k(\A)}=0$ and $G|_{B^k(\A)}=(d|_{L^{k-1}(\A)})^{-1}$.
      Define a sequence of linear maps $\lambda_r$ of degree $2-k$ as follows.
      There is no map $\lambda_1$ but we define the composition $G\lambda_1$ by $-\id_{\A}$.
      The map $\lambda_2$ is the same as the multiplication of $\A$.
      And for $r \geq 3$, we inductively define $\lambda_r$ by 
      \[\lambda_r=\sum_{t=1}^{r-1} (-1)^{t+(r-t)(\sum_{j=1}^t a_j)+1}\lambda_2(G\lambda_{r-t}(a_r,\dots,a_{r-t}),G\lambda_t(a_t,\dots,a_1))\]
      for $a_1,\dots,a_r \in \A$.
      Let $p:\A \to \HH^*(\A)$ and $i:\HH^*(\A) \to \A$ be the canonical projection and injection, respectively.
      Then $\HH^*(\A)$ is an $A_\infty$-algebra with multiplications $m_r=p \lambda_r i^{\otimes r}: \HH^*(\A)^{\otimes r} \to \HH^*(\A)$.
    \end{rmk}

    \begin{thm}[\cite{LPWZ} Theorem A]\label{LPWZ}
      Let $E=\bigoplus_{k \geq 0}E_k$ be a graded algebra with finite dimensional spaces $E_k$ for any $k \geq 0$.
      Let $\mathfrak{m}=\bigoplus_{k \geq 1}E_k$  and $Q=\mathfrak{m}/\mathfrak{m}^2$.
      Let $R=\bigoplus_{k\geq 2} R_k$ be a minimal graded space of relations of $E$, with $R_k$ chosen so that
      \[
        R_k \subset \bigoplus_{1 \leq l \leq k-1} Q_l \otimes E_{k-l} \subset \left( \bigoplus_{r \geq 2} Q^{\otimes r} \right) _k.
      \]
      For each $t \geq 2$ and $k \geq 2$, let $i_k:R_k \to (\bigoplus_{r \geq 2}Q^{\otimes r})_k$ be the inclusion map and let $i_k^t$ be the composite
      \[
        R_k \overset{i_k}{\longrightarrow} \left( \bigoplus_{r \geq 2} Q^{\otimes r} \right)_k \longrightarrow (Q^{\otimes t})_k.
      \]
      Then there is a choice of $A_\infty$-category $(\HH^*,m)$ for $\HH^*=\Ext^*(S_E,S_E)$ so that in any degree $-k$, the multiplication $m_t$ of $\HH^*$ restricted to $(\HH^1)_{-k}^{\otimes t}$ is equal to the map
      \[
        D(i_k^t):((\HH^1)^{\otimes t})_{-k} = D((Q^{\otimes t})_k) \longrightarrow D(R_k) \subset H_{-k}^2.
      \]
    \end{thm}
    For simplicity we forget the grading of $E$, i.e. assume that $E$ is just a finite dimensional algebra and $\mathfrak{m} = \rad A$.
    Then this theorem implies that the relations of $E$ induces the $A_\infty$-multiplication on $\HH^1$.
    Of course, the algebra generated by $D(\HH^1)$ with relations $D(m_t(\HH^1)^{\otimes t})$ for $t \geq 2$ is Morita equivalent to $E$.

    \begin{lem}\label{b}
      Let $A$ be a $K$-algebra and
      \[\xymatrix{\mbP_M:\cdots \ar[r]^-{\partial_3}& P_2 \ar[r]^-{\partial_2}& P_1 \ar[r]^-{\partial_1} & P_0 & \text{and} &
      \mbP_N:\cdots \ar[r]^-{\partial'_3}& P'_2 \ar[r]^-{\partial'_2}& P'_1 \ar[r]^-{\partial'_1} & P'_0}\]
      be projective resolutions of some $A$-modules $M$ and $N$.
      Assume that $k \geq 1$ and that a chain morphism $f=(f^{(l)})_{l\in\Z}:\mbP_M \to \mbP_N[k]$
      \[\xymatrix{
        \cdots\ar[r] & P_{k+2}\ar[r]^-{\partial_{k+2}}\ar[d]^-{f^{(k
        +2)}} & P_{k+1}\ar[r]^-{\partial_{k+1}}\ar[d]^-{f^{(k+1)}} & P_k\ar[r]\ar[d]^-{f^{(k)}} & \cdots\ar[r] & P_1\ar[r]^-{\partial_1} & P_0\\
        \cdots\ar[r] & P'_2\ar[r]_-{(-1)^k\partial'_2} & P'_1\ar[r]_-{(-1)^k\partial'_1} & P'_0
      }\]
      satisfies $f^{(k)}=\partial'_1u^{(k)}$ for some $A$-homomorphism $u^{(k)}:P_k \to P'_1$.
      Then the morphism $f$ is null-homotopic.
      In particular, a chain morphism $f$ is uniquely determined by $f^{(k)}$, up to homotopy. 
    \end{lem}
    \begin{proof}
      We will show that for each $l\geq k$, there are $A$-homomorphisms $u^{(l)}:P_l \to P'_{l-k+1}$ and $u^{(l-1)}:P_{l-1} \to P'_{l-k}$ such that $f^{(l)}=\partial'_{l-k+1}u^{(l)} - (-1)^{k-1} u^{(l-1)}\partial_l$.
      We proceed by induction on $l$.
      If $l=k$, then we have $u^{(k)}:P_k \to P'_1$ such that $f^{(k)}=\partial'_1u^{(k)}$ by our assumption.
      Assume that $l>k$ and that there are $u^{(l-1)}:P_{l-1} \to P'_{l-k}$ and $u^{(l-2)}:P_{l-2} \to P'_{l-k-1}$ such that $f^{(l-1)}=\partial'_{l-k}u^{(l-1)} - (-1)^{k-1} u^{(l-2)}\partial_{l-1}$.
      Then we have equalities
      \[\begin{split}
        \partial'_{l-k}f^{(l)}
        &=(-1)^k f^{(l-1)}\partial_k\\
        &=(-1)^k (\partial'_{l-k}u^{(l-1)} - (-1)^{k-1} u^{(l-2)}\partial_{l-1})\partial_l\\
        &=(-1)^k \partial'_{l-k}u^{(l-1)}\partial_l.
      \end{split}\]
      Hence $\image(f^{(l)}-(-1)^k u^{(l-1)}\partial_l) \subset \kernel\partial'_{l-k}=\image\partial'_{l-k+1}$.
      Since $P_l$ is projective, there exists $u^{(l)}:P_l\to P'_{l-k+1}$ such that $f^{(l)}-(-1)^k u^{(l-1)}\partial_l=\partial'_{l-k+1}u^{(l)}$.
      Thus we conclude that $f^{(l)}=\partial'_{l-k+1}u^{(l)} - (-1)^{k-1} u^{(l-1)}\partial_l$.
    \end{proof}

    Fix $1 \leq i \leq n$.
    Define $E$ as the opposite algebra of $\End_A(\Delta(i))$ and let $\A_i$ be a full subcategory of $\A$ whose object is $i$.
    The following lemmas show that we can choose bases of $\HH^k(\A_i)$, $B^k(\A_i)$ and $L^{k-1}(\A_i)$ just by concentrating on their $k$-th components.

    \begin{lem}\label{proj resol}
      A chain morphism $g \in Z^k(\A_i)$ is in $B^k(\A_i)$ if and only if $g^{(k)}\in\rad_A(P_k,P_0)$.
    \end{lem}
    \begin{proof}
      By \Cref{standard modules}, we have $\image \partial_1 \cong \sum\image\varphi$ for $j\geq i$ and $\varphi\in\rad_A(P(i),P(j))$.
      So if $g^{(k)}\in\rad_A(P_k,P_0)$, there exists $u^{(k)}:P_k \to P_1$ such that $g^{(k)}=\partial_1u^{(k)}$.
      Apply the previous lemma, we conclude that $g$ is null-homotopic.
      On the other hand, suppose $g^{(k)} \notin \rad_A(P_k,P_0)$.
      Since differentials $\partial_l$ of $\mbP_i$ are in $\rad_A(P_l,P_{l-1})$ for any $l\geq0$, compositions $\partial_1 u^{(k)}$ and $u^{(k-1)} \partial_k$, and hence their sum, must be in $\rad_A(P_l,P_{l-1})$ where $u^{(k-1)}:P_{k-1} \to P_0$ and $u^{(k)}:P_k \to P_1$.
      Hence $g^{(k)} \neq \partial_1 u^{(k)} - (-1)^{k-1}  u^{(k-1)} \partial_k$ for any $u^{(k-1)}$ and $u^{(k)}$, and this implies that $g$ is not null-homotopic.
    \end{proof}

    \begin{lem}\label{basis of Ext} 
      \begin{enumerate}
        \item For a projective resolution       
          \[
            \xymatrix{
              \mbP_i: \cdots \ar[r]^-{\partial_3} & P_2 \ar[r]^-{\partial_2} & P_1 \ar[r]^-{\partial_1} & P_0
            }
          \]
          of $\oDelta(i)$, we write $P_k=P(i)^{\oplus c_k}\oplus \overline{P_k}$
          for $k\geq 1$, where $c_k\geq 0$ and $\overline{P_k}$ dose not include $P(i)$ as direct summands.
          Then the chain morphisms $f_l:\mbP_i \to \mbP_i[k]$ with $f_l^{(k)}=[\pi_l,0]:P(i)^{\oplus c_k}\oplus \overline{P_k} \to P_0$ form a basis of $\Ext^k(\oDelta(i),\oDelta(i))$ for $1\leq l \leq c_k$.
          Here, $\pi_l:P(i)^{\oplus c_k} \to P(i)$ are the canonical projections on $l$-th components.
        \item We can choose, as a basis of $L^{k-1}(\A_i)$, non-chain morphisms $u_l \in L^{k-1}(\A_i)$ for $u_l^{(k-1)}=0$ and a subset $\{u_l^{(k)}\}_l$ of a basis of $\Hom_A(P_k,P_1)$ satisfying $\partial_1 u_l^{(k)} \neq 0$.
      \end{enumerate}
    \end{lem}
    \begin{proof}
      For the claim 1.
      If $c_k=0$, then any morphisms $f \in Z^k(\A_i)$ satisfies $f^{(k)} \in \rad_A(P_k,P_0)$.
      Hence $f \notin \HH^k(\A_i)$ by \Cref{proj resol}, and so $\HH^k(\A_i)=0$.
      We may assume $c_k \geq 1$.
      Let $f_l:\mbP_i \to \mbP_i[k]$ be in $Z^k(\A_i)$ satisfying $f_l^{(k)}=[\pi_l,0]:P(i)^{\oplus c_k}\oplus \overline{P_k} \to P_0$ for some $1 \leq l \leq c_k$.
      Then $f_l^{(k)}$ is not in $\rad_A(P_k,P_0)$.
      Hence by \Cref{proj resol}, the chain morphism $f_l$ is not in $B^k(\A_i)$.
      Thus, it is a non-zero element of $\HH^k(\A_i)=\Ext^k_A(\oDelta,\oDelta)$.
      Moreover, it immediately turns out $f_1,\dots,f_{c_k}$ form a basis of $\HH^k(\A_i)=\Ext^k_A(\oDelta,\oDelta)$.
      For the claim 2.
      We first show that any $g \in B^k(\A_i)$ can be written by $d(u)$ for some $u \in L^{k-1}(\A_i)$ such that $u^{(k-1)}=0$.
      By \Cref{proj resol} again, $g^{(k)} \in \rad_A(P_k,P_0)$, and hence $\image g^{(k)} \subset \image \partial_1$, which is from the definition of $\oDelta(i)$, see \Cref{standard modules}.
      By the projectivity of $P_k$, there exits $u^{(k)}:P_k \to P_1$ such that $g^{(k)} = \partial_1 u^{(k)}$.
      Hence we can take $u \in L^k(\A_i)$ as $u^{(k-1)}=0$.
      Moreover if $\partial_1 u^{(k)}=0$, we have $d(u)^{(k)}=0$.
      So any non-zero $u \in L^k(\A_i)$ satisfies $\partial_1 u^{(k)} \neq 0$.
    \end{proof}
    
    \begin{prop}\label{mult in cycles}
      For $a_1,\dots,a_r \in \A_i^1$, we have $\lambda_r(a_r,\dots,a_1) \in Z^2(\A_i)$.
      Moreover, let $a_1,\dots,a_r \in \HH^1(\A_i)$.
      Then $\lambda_r(a_r,\dots,a_1)\in \HH^2(\A_i)$ if and only if $(a_r \circ G\lambda_{r-1}(a_{r-1},\dots,a_1))^{(2)}$ is surjective if and only if for 
      \[(G\lambda_{r-1}(a_{r-1},\dots,a_1))^{(2)} = \sum h_j \in \Hom_A(P_2,P(i))^{\oplus c_1} \oplus {\rm Hom}_A(P_2,\overline{P_1}),\]
      there exists $h_j \in \Hom_A(P_2,P(i))$ is surjective.
    \end{prop}
    \begin{proof}
      By the definition of $\lambda_r$ in \Cref{Mer}, we have
      \[\begin{split}
        & d(\lambda_r(a_r,\dots,a_1))\\
        =& \sum_{t=1}^{r-1} (-1)^{t(r-t)+1} dG\lambda_{r-t}(a_k,\dots,a_
        {t+1}) \circ G\lambda_t(a_t,\dots,a_1)\\
        +& (-1)^{|G\lambda_{r-t}(a_r,\dots,a_{t+1})|} \sum_{t=1}^{r-1} (-1)^{t(r-t)+1} G\lambda_{r-t}(a_r,\dots,a_{t+1}) \circ dG\lambda_t(a_t,\dots,a_1).
      \end{split}\]
      On the first term, we have 
      \[\begin{split}
        & \sum_{t=1}^{r-1} (-1)^{t(r-t)+1} dG\lambda_{r-t}(a_r,\dots,a_{t+1}) \circ G\lambda_t(a_t,\dots,a_1)\\
        =& \sum_{t=1}^{r-1} (-1)^{t(r-t)+1} \lambda_t(a_r,\dots,a_{r-t+1}) \circ G\lambda_{r-t}(a_{r-t},\dots,a_1)\\
        =& \sum_{t=1}^{r-1} \sum_{u=1}^{t-1} (-1)^\sigma G\lambda_u(a_r,\dots,a_{r-u+1}) \circ G\lambda_{t-u}(a_{r-u},\dots,a_{r-t+1}) \circ G\lambda_{r-t}(a_{r-t},\dots,a_1),
      \end{split}\]
      where $\sigma=t(r-t)+u(t-u)$.
      Since $|G\lambda_{r-t}(a_r,\dots,a_{t+1})|=1$, the second term equals to
      \[\begin{split}
        & -\sum_{t=1}^{r-1} (-1)^{t(r-t)+1} G\lambda_{r-t}(a_r,\dots,a_{t+1}) \circ \lambda_t(a_t,\dots,a_1)\\
        =& -\sum_{t=1}^{r-1}\sum_{u=1}^{t-1} (-1)^\sigma  G\lambda_{r-t}(a_r,\dots,a_{t+1}) \circ G\lambda_{t-u}(a_t,\dots,a_{u+1}) \circ G\lambda_u(a_u,\dots,a_1).
      \end{split}\]
      Since $\sigma=u(r-t)+(t-u)(r-t)+u(t-u)$, the signs of each term in the two polynomials coincide.
      Hence $d(\lambda_r(a_r,\dots,a_1))=0$, that is, $\lambda_r(a_r,\dots,a_1) \in Z^2(\A_i)$.

      Next, assume $a_1,\dots,a_r \in \HH^1(\A_i)$ and let $f=\lambda_r(a_r,\dots,a_1)$.
      If $f^{(2)}$ is surjective, so are at least one of morphisms
      \[(G\lambda_{r-t}(a_r,\dots,a_{t+1}) \circ G\lambda_t(a_t,\dots,a_1))^{(2)} =
        (G\lambda_{r-t}(a_r,\dots,a_{t+1}))^{(1)} \circ (G\lambda_t(a_t,\dots,a_1))^{(2)}.
      \]
      Since $G\lambda_{r-t}(a_r,\dots,a_{t+1})$ are in $L^1(\A_i)$ for $1 \leq t \leq r-2$, we have $G\lambda_{r-t}(a_r,\dots,a_{t+1})^{(1)}=0$.
      Hence
      \[\left(\sum_{t=1}^{r-1} (-1)^{t(r-t)+1} G\lambda_{r-t}(a_r,\dots,a_{t+1}) \circ G\lambda_t(a_t,\dots,a_1)\right)^{(2)} = ((-1)^r a_r \circ G\lambda_{r-1}(a_{r-1},\dots,a_1))^{(2)}.\]
      This shows that $f^{(2)}$ is surjective if and only if $(a_r \circ G\lambda_{r-1}(a_{r-1},\dots,a_1))^{(2)}$ is so.
      Moreover since $a_r^{(1)}$ is a surjection onto $P(i)$, the morphism $(G\lambda_{r-1}(a_{r-1},\dots,a_1))^{(2)}$ has a summand $h_j \in \Hom_A(P_2,P(i)) \subset \Hom_A(P_2,P(i))^{\oplus c_1} \oplus \Hom_A(P_2,\overline{P_1})$, which is surjective.
    \end{proof}

    Now, let $n_j=[{}_AA:P(j)]$ and $\{e^A_j\}_{1 \leq j \leq n}$ be a set of pairwise orthogonal idempotents of $A$ with $Ae^A_j \cong P(j)^{\oplus n_j}$.
    Consider the factor algebra $A'=A/AeA$ for $e=\sum_{j=i+1}^n e^A_{j}$.
    Then $\oDelta(i)$ and $\Delta(i)$ are also $A'$-modules because $e\oDelta(i)=0$ and $e\Delta(i)=0$.
    We will write $\bigoplus_{k \geq 0} \Hom_{A'}(\mbP'_i,\mbP'_i[k])$ by $\A_i'$, where 
    \[\xymatrix{\mbP'_i:\cdots \ar[r] & P'_2 \ar[r] & P'_1 \ar[r] & P'_0}\]
    is a projective resolution of $\oDelta(i)$ in $\module A'$.
    Of course, $\A_i'$ is a differential category whose differential $d'=d|_{\A_i'}$ is the restriction of that of $\A_i$.
    Moreover there are multiplications $m'_k$ such that $(\HH(\A_i'),m')$ is an $A_\infty$-category by \Cref{(4.3)}.
    Further the maps $\lambda'_r:\A_i'^{\otimes r} \to \A_i'$ is defined by a way similar to one in \Cref{Mer}.
    Now numbers $c_k$ coincide with the numbers of copies of $P_{A'}(i)$ in $P'_k$, because the numbers $c_k$ independent of $P_A(i+1),\dots,P_A(n)$, or $\oDelta(i+1),\dots,\oDelta(n)$.
    Moreover the previous lemmas can be argued by ignoring all $\overline{P_k}$, hence by replacing $A$ with $A'$.
    We get the following fact.
    
    \begin{prop}\label{AA' iso}
      There is an isomorphism $\HH(\A_i) \cong \HH(\A_i')$ as $\Z$-graded categories.
      Moreover for $a_1,\dots,a_r \in \HH^1(\A_i) \cong \HH^1(\A_i')$, $\lambda_r(a_r,\dots,a_1)\in \HH^2(\A_i)$ if and only if $\lambda'_r(a_r,\dots,a_1)\in \HH^2(\A_i')$.
    \end{prop}
    \begin{proof}
      We notice that the $\Z$-graded categories $\HH(\A_i)$ and $\HH(\A_i')$ are isomorphic.
      Indeed, \Cref{basis of Ext} showed us that a $K$-basis of $\HH^k(\A_i)$ are formed by chain morphisms $f:\mbP_i \to \mbP_i[k]$ with the $k$-th component $f^{(k)} = [\pi_l,0]$.
      Moreover the $k$-th component $P'_k$ of the resolution $\mbP'_i$ is isomorphic to $P_{A'}(i)^{\oplus c_k}$.
      Assign $f$ to $f':\mbP'_i \to \mbP'_i[k]$ with $f'_k=\pi'_l:P_{A'}(i)^{\oplus c_k} \to P_{A'}(i)$ the canonical projection, then we get an isomorphism from $\HH(\A_i)$ to $\HH(\A_i')$ as $\Z$-graded categories.

      Next we show that they have also the relationship on their $\A_\infty$-multiplications when we restrict them to $\HH^1(\A_i)$ and $\HH^1(\A_i')$.
      By \Cref{proj resol,basis of Ext}, we may fix bases of $\HH^k(\A_i)$, $B^k(\A_i)$ and $L^k(\A_i)$ as follows.
      \begin{description}
        \item[$\HH^k(\A_i)$:] chain morphisms $f_l$'s for $1 \leq l \leq c_k$ with $f_l^{(k)}=[\pi_l,0]$
        \item[$B^k(\A_i)$:] chian morphisms $g_l$'s for a basis $\{g_l^{(k)}\}_l$ of $\rad_A(P_k,P_0)$.
        \item[$L^k(\A_i)$:] non-chian morphisms $u_l$'s for $u^{(k)}=0$ and a subset $\{u_l^{(k+1)}\}_l$ of a basis of $\Hom_A(P_{k+1},P_1)$ satisfying $\partial_1 u_l^{(k+1)} \neq 0$.
      \end{description}
      By \Cref{mult in cycles}, $\lambda_r(a_r,\dots,a_1)\in \HH^2(\A_i)$ if and only if for $(G\lambda_{r-1}(a_{r-1},\dots,a_1))^{(2)} = \sum h_j \in {\rm Hom}_A(P_2,P(i))^{\oplus c_1} \oplus {\rm Hom}_A(P_2,\overline{P_1})$, there exists $h_j \in {\rm Hom}_A(P_2,P(i))$ is surjective.
      And the latter condition depends on only automorphisms between $P(i)$.
      Thus the multiplications of $\HH^1(\A_i)$ is calculated by that of $\Aut_A(P_A(i))$, and the arguments above can be done similarly for $\A_i'$.
      Moreover we have isomorphisms $\Aut_A(P_A(i)) \cong \Aut_{A'}(P_{A'}(i))$ as algebras.
      Hence there is also a one-to-one correspondence between the multiplications in $\HH^1(\A_i)$ and in $\HH^1(\A_i')$.
    \end{proof}

    \begin{prop}\label{A'E iso}
      There is an isomorphism $\HH(\A_i') \cong \Ext^*_E(S_E,S_E)$ as $A_\infty$-categories.
    \end{prop}
    \begin{proof}
      We have $\Hom_{A'}(\Delta(i),\oDelta(i)) \cong S_E$ and $\Hom_{A'}(\Delta(i),\Delta(i)) \cong P_E$ as $E$-modules.
      Moreover since $\Delta(i) \cong P_{A'}(i)$ is a projective $A'$-module, the functor $\Hom_{A'}(\Delta(i),-):\module A' \to \module E$ is exact.
      Hence, the complex 
      \[\xymatrix{\mbP_E=\Hom_{A'}(\Delta(i),\mbP'_i):\cdots \ar[r] & P_E^{\oplus c_2} \ar[r] & P_E^{\oplus c_1} \ar[r] & P_E^{\oplus c_0}}\] 
      is a projective resolution of $S_E$.
      So there is an isomorphism between $\bigoplus_{k\geq1}\Hom_E(\mbP_E,\mbP_E[k])$ and $\A_i'=\bigoplus_{k\geq1}\Hom_{A'}(\mbP'_i,\mbP'_i[k])$ as differential graded categories.
      This implies that $\Ext^*_E(S_E,S_E) \cong \\ \HH(\A_i')$ as $A_\infty$-categories.
    \end{proof}

    These propositions imply that $e_iBe_i$ is finite dimensional as follows.

    \begin{thm}\label{EB Morita equiv}
      Let $E=\End_A(\Delta(i))$ and $\B_A=(B,W)$ be the bocs induced from $A$.
      Then $E$ and $e_iBe_i$ are Morita equivalent.
      In particular, $e_iBe_i$ is finite dimensional.
    \end{thm}
    \begin{proof}
      By the construction of $B$, the algebra $e_iBe_i$ is generated by a $K$-basis of $Q_0(i,i)=\\Ds(\HH^1(\A_i))$ and has relations induced from $Q_{-1}(i,i)=Ds(\HH^2(\A_i))$.
      Similarly, \Cref{LPWZ} implies that a $K$-basis of $D\Ext^1_E(S_E,S_E)$ generates $E$ and that a $K$-basis of $D\Ext^2_E(S_E,S_E)$ constitutes relations of $E$.
      Next, identify morphisms in $\HH^1(\A_i)$ and $\Ext^1_E(S_E,S_E)$. 
      Then the previous propositions show that the multiplications of morphisms in $\HH^1(\A_i)$ is included in $\HH^2(\A_i)$ if and only if that is in $\Ext^2_E(S_E,S_E)$.
      Since these multiplications as $A_\infty$-categories are induced from the multiplications of the algebras $e_iBe_i$ and $E$, these algebras are Morita equivalent.
    \end{proof}

    The following lemma shows that the bocs $\B_A=(B,W)$ is one-cyclic directed.
    And this bocs is what we want i.e. it satisfies that $\F(\oDelta_A) \simeq \module \B_A$, which is proved in \Cref{thm1} below.

    \begin{lem}\label{one-cyclic directed}
      Let $A$ be a $\oDelta$-filtered algebra.
      Then the bocs $\B_A$ given above is one-cyclic directed.
    \end{lem}
    \begin{proof}
      The conditions 
      \[\End_A(\oDelta(i))\cong K,\ \Hom_A(\oDelta(i),\oDelta(j))=0 \text{ and } \Ext^1_A(\oDelta(i),\oDelta(j))=0 \text{ for } i>j\] implies that 
      \[\overline{W} \cong \bigoplus_{i>j}(Be_i \tensor e_jB)^{d_{ij}} \text{ and }  \Ext^1_B(S(i),S(j))=0 \text{ for } i>j.\]
      The second condition shows $\rad_B(P_B(j),P_B(i))/\rad^2_B(P_B(j),P_B(i))=0$ for $j<i$, and hence\\ $\rad_B(P_B(j),P_B(i))=0$ for $j<i$.
    \end{proof}

  \begin{thm}\label{thm1}
    Let $A$ be a $\oDelta$-filtered algebra.
    Then the one-cyclic directed bocs $\B_A=(B,W)$ constructed above satisfies $\module \B_A \simeq \F(\oDelta_A)$.
  \end{thm}
  \begin{proof}
    Let $\HH^*(\A)$ be an $A_\infty$-category over $\{1,2,\dots,n\}$ such that $\HH^k(\A)(i,j)=\Ext^k_A(\oDelta,\oDelta)$.
    We already claimed $\F(\oDelta) \simeq \HH^0(\twmod \HH^*(\A))$ in \Cref{Thm 5.4}.
    So it suffices to give an equivalence $F:\HH^0(\twmod \HH^*(\A)) \to \module\B_A$ for the bocs $\B_A$ given above.
    For objects, we confirmed a one-to-one correspondence between $(X,\delta)$ and $\mathfrak{X}_{\delta}$ in \Cref{Prop 6.3}.
    For morphisms, $F$ assigns $f \in \HH^0(\twmod \HH^*(\A))$ to $M(sf) \in \bigoplus_{1 \leq i,j \leq n}\Hom_K(\QQ_1(i,j),\Hom_K(X(i),Y(j)))$ such that $M(sf)(\varphi_k)=0$ for $\varphi_k \in \QQ_1 \cap d(B)$.
    Although the later is not a morphism of $\module\B_A$, we can identify them by the following isomorphisms as vector spaces
    \[\begin{split}
        &  \Hom_{\B}(X,Y)\\
        =& \Hom_{B \tensor B^{\op}}(W,\Hom_K(X,Y))\\
        =& \Hom_{B \tensor B^{\op}}((\bigoplus_k B \varphi_k B)/d(B),\Hom_K(X,Y))\\
        \cong& \{\Phi \in \Hom_{B \tensor B^{\op}}(\bigoplus_k B \varphi_k B,\Hom_K(X,Y)) \mid \Phi(d(B))=0\}\\
        \cong& \{g \in \bigoplus_{1 \leq i,j \leq n}\Hom_K(\QQ_1(i,j),\Hom_K(X(i),Y(j))) \mid \Phi_g(d(B))=0\},
      \end{split}\] 
    where $\varphi_k$ are generators of $\QQ_1$ over $K$ and $\Phi_g$ is given by $\Phi_g(b\varphi_kb')=bg(\varphi_k)b'$.
    Then $F$ is an equivalence.
    We omit to check that here since the calculation is the same to that in \cite{KKO}.
  \end{proof}


  Thus we give a new construction of one-cyclic directed bocses $\B_A$ from $\oDelta$-filtered algebras $A$ with $\F(\oDelta_A) \cong \module \B_A$, which is generalization of the method in \cite{KKO}.
  Indeed, if $A$ is quasi-hereditary, the bocs constructed by our way is the same to that by using the methods in \cite{KKO}.

  \subsection{$\oDelta$-filtered algebras from one-cyclic directed bocses}\label{pfa form pdb}
    In this subsection, we construct an algebra $R$ from a one-cyclic directed bocs $\B$ with $\F(\Delta_R)\simeq \module\B$.
    Since Burt and Butler's theory are confirmed for bocses with projective kernels, we can apply them to one-cyclic directed bocses.
    By immitating the arguments in \cite{KKO}, we will show that the right Burt-Butler algebra of a one-cyclic directed bocs is a $\oDelta$-filtered algebra.

    \begin{thm}\label{thm2}
      Let $\B=(B,W)$ be a one-cyclic directed bocs.
      Then its right Burt-Butler algebra $R$ of $\B$ is a $\oDelta$-filtered algebra with a homological proper Borel subalgebra $B$ such that $\F(\oDelta_R) \simeq \module \B$.
    \end{thm}
    \begin{proof}
      Remark that the algebra $R$ is basic since so is $B$.
      Put $\oDelta_R(i) = R \tensor S_B(i)$.
      Recall that ${\rm Ind}(B,R)$ is a full subcategory of $\module R$ whose objects are of the form $R \tensor X$ for $B$-modules $X$.
      By the definition of $\oDelta_R$ in this proof, we have an equivalence 
      \[T:{\rm Ind}(B,R) \to \F(\oDelta_R) \text{ with } T(R \tensor S_B(i)) \cong \oDelta_R(i).\]
      Indeed, obviously this functor is dense and ${\rm Ind}(B,R)$ and $\F(\oDelta_R)$ are full subcategories of $\module R$.
      On the other hand, we also have an equivalence 
      \[R\tensor_B-:\module \B \to {\rm Ind}(B,R)\]
      by \Cref{BB} (\ref{BB Thm 2.5}).
      Hence we obtain an equivalence $T \circ (R\tensor_B-): \module \B \to \F(\oDelta_R)$.
      Next we show that $R$ is a $\oDelta$-filtered algebra.
      Now $\oDelta_R(i)$ is a factor module of $P_R(i)$ because there is a surjection $P_R(i) \to \oDelta_R(i)$ induced from the canonical epimorphism $P_B(i) \to S_B(i)$.
      By the equivalence $T \circ (R\tensor_B-)$, we have an isomorphism as $K$-vector spaces
      \[\begin{split}
        & \Hom_R(P_R(i),\oDelta_R(j))\\
        =& \Hom_R(R\tensor_BP_B(i),R\tensor_BS_B(j))\\
        \cong& \Hom_{\B}(P_B(i),S_B(j))\\
        \cong& \Hom_B(W \tensor_B P_B(i),S_B(j))\\
        \cong& \Hom_B((B \oplus \overline{W}) \tensor_B P_B(i),S_B(j)) \\
        \cong& \Bigl(\Hom_B(B \tensor_B P_B(i),S_B(j))\Bigr)
        \oplus \left(\Hom_B(\bigoplus_{k>l}(Be_k \tensor_K e_lB)^{d_{kl}} \tensor_B P_B(i),S_B(j))\right) \\
        \cong& \Bigl(\Hom_B(P_B(i),S_B(j))\Bigr)
        \oplus \left(\Hom_B(\bigoplus_{j>l}(P_B(j) \tensor_K e_l B e_i)^{d_{jl}} ,S_B(j))\right) \\
        \cong& \Bigl(\Hom_B(P_B(i),S_B(j))\Bigr)
        \oplus \left(\Hom_B(\bigoplus_{j>l}P_B(j)^{d_{jl} \dim e_l B e_i} ,S_B(j))\right).
      \end{split}\]
      As a result, we have
      \[
        \dim\Hom_R(P_R(i),\oDelta_R(j))=
        \begin{cases}
          1+ \sum_{j>l} d_{jl}\dim e_l B e_i &(i\leq j)\\
          0 &(i>j),
        \end{cases}
      \]
      and hence $\oDelta_R(j)$ has the composition factor only of the form $S(i)$ with $i \leq j$ and $[\oDelta_R(j):S_R(j)]=1$.

      Finally, by using \Cref{BB} (\ref{BB Thm 3.8}), we have $\dim\Ext^1_R(\oDelta_R(i),\oDelta_R(j)) \leq \dim\Ext^1_B(S_B(i),S_B(j)) =0$ for $i > j$.
      On the other hand, $R \cong R \tensor B$ is filtered by $\oDelta_R$.
      The above conditions show that $\oDelta_R$ is the properly standard system over $R$.
      Hence the algebra $R$ is $\oDelta$-filtered.
      Moreover, the functor $R\tensor_B-$ is exact, by \Cref{BB} (\ref{BB Sec 2}).
      Thus $B$ is a homological proper Borel subalgebra of $R$.
    \end{proof}

  \subsection{Relation between $\oDelta$-filtered algebras and $\F(\oDelta)$}\label{pfa and filt}
    In this subsection, we show that $\oDelta$-filtered algebras are uniquely determined by $\F(\oDelta)$, up to Morita equivalence.
    But we can not prove this proposition similarly to \cite{DR} Theorem 2.
    This is because a set $\{\overline{\Theta}_1,\dots,\overline{\Theta}_n\}$ of objects, satisfying the homological properties of a properly standard system, of an abelian $K$-category $\Cat$ may have self-extensions.
    It may require that we take extensions of objects of $\F(\overline{\Theta}_i)$ by $\overline{\Theta}_i$ repeatedly when constructing $\Ext$-projective object $P_{\overline{\Theta}}(i)$.
    In order for such a situation not to occur, $\Cat$ will be assumed to be the category $\module C$ for some finite dimensional algebra $C$.
    Fortunately, we just need a one-to-one correspondence between Morita equivalence classes of $\oDelta$-filtered algebras and equivalence classes of categories of modules with $\oDelta$-filtrations. 
    So we give the next theorem given in \cite{ADL} Theorem 2.3.

    \begin{thm}[\cite{ADL} Theorem 2.3]\label{ADL 2.3}
      Let $C$ be a finite dimensional algebra.
      Then there exists a $\oDelta$-filtered algebra $A$, unique up to Morita equivalence, such that the category $\F(\oDelta_C)$ and $\F(\oDelta_A)$ are equivalent.
      In particular, $\oDelta$-filtered algebras $A$, and $C$ are Morita equivalent if and only if there exists an equivalence $F:\F(\oDelta_A) \to \F(\oDelta_C)$.
    \end{thm}

  \subsection{Main result}\label{main result}  
    As the final subsection of this article, we sum up the previous three subsections and get the following theorem giving a generalization of \Cref{KKO main}.
    \begin{thm}\label{main2}    
      We have a bijection
      \begin{gather*}
        \{\text{Morita equivalence classes of $\oDelta$-filtered algebras}\}\ \\
        \updownarrow \\
        \{\text{Equivalence classes of the module categories over one-cyclic directed bocses}\}.
      \end{gather*}
      Let a $\oDelta$-filtered algebra $A$ and a one-cyclic directed bocs $\B=(B,W)$ correspond via the above bijection.
      Then the right Burt-Butler algebra $R_{\B}$ of $\B$ is Morita equivalent to $A$.
      Moreover, $R_{\B}$ has a homological proper Borel subalgebra $B$.
    \end{thm}
  \begin{proof}
    The second half condition is shown in \Cref{thm2}.
    On the first half condition, let $A$ be a $\oDelta$-filtered algebra.
    Then we can construct a one-cyclic directed bocs $\B=\B_A$ given in \Cref{odb form pfa}.
    Moreover, the right Burt-Butler algebra $R$ of $\B$ is a $\oDelta$-filtered algebra by the discussion in \Cref{pfa form pdb}.
    In these cases, there are equivalences 
    \[\F(\oDelta_A) \cong \module\B \cong \F(\oDelta_R)\]
    by \Cref{thm1,thm2}.
    Finally, \Cref{ADL 2.3} guarantees $A$ and $R$ are Moreita equivalent.
  \end{proof}

  The following theorem is the result for $\Delta$-filtered algebras.
  Actually, almost the same arguments in \cite{KKO} give its proof.
  And we remark that in \cite{BPS} Theorem 12.9, Bautista, P\'erez and Salm\'eron gave the result stronger than this.

  \begin{thm}\label{main1}    
    We have a bijection
    \begin{gather*}
      \{\text{Morita equivalence classes of $\Delta$-filtered algebras}\}\ \\
      \updownarrow \\
      \{\text{Equivalence classes of the module categories over weakly directed bocses}\}.
    \end{gather*}
    Let a $\Delta$-filtered algebra $A$ and a weakly directed bocs $\B=(B,W)$ correspond via the above bijection.
    Then the right Burt-Butler algebra $R_{\B}$ of $\B$ is Morita equivalent to $A$.
    Moreover, $R_{\B}$ has a homological exact Borel subalgebra $B$.
  \end{thm}

\end{document}